%% file: main.tex
\documentclass[a4paper]{amsart}

\newtheorem{theorem}{Theorem}[section]
\newtheorem{lemma}[theorem]{Lemma}

\theoremstyle{definition}
\newtheorem{definition}[theorem]{Definition}

\theoremstyle{remark}
\newtheorem{remark}[theorem]{Remark}

\numberwithin{equation}{section}

\newcommand{\abs}[1]{\lvert#1\rvert}

\usepackage{AMS_style_file}

\begin{document}
\title{A Deep Uzawa-Lagrange Multiplier Approach for Boundary
  Conditions in PINNs and Deep Ritz Methods}
\author{Charalambos G. Makridakis}
\address{DMAM, University of Crete / Institute of Applied and Computational Mathematics, FORTH, 700 13 Heraklion, Crete, Greece, and MPS, University of Sussex, Brighton BN1 9QH, UK}
\email{C.G.Makridakis@iacm.forth.gr}
\thanks{}

\author{Aaron Pim}
\address{Department of Mathematical Sciences,
      University of Bath, Claverton Down, Bath BA2 7AY, UK}
\curraddr{}
\email{A.R.Pim@bath.ac.uk}
\thanks{}

\author{Tristan Pryer}
\address{Department of Mathematical Sciences,
      University of Bath, Claverton Down, Bath BA2 7AY, UK}
\curraddr{}
\email{tmp38@bath.ac.uk}
\thanks{}

\begin{abstract}
  We introduce a deep learning-based framework for weakly enforcing
  boundary conditions in the numerical approximation of partial
  differential equations. Building on existing physics-informed neural
  network and deep Ritz methods, we propose the Deep Uzawa algorithm,
  which incorporates Lagrange multipliers to handle boundary
  conditions effectively. This modification requires only a minor
  computational adjustment but ensures enhanced convergence properties
  and provably accurate enforcement of boundary conditions, even for
  singularly perturbed problems.

  We provide a comprehensive mathematical analysis demonstrating the
  convergence of the scheme and validate the effectiveness of the Deep
  Uzawa algorithm through numerical experiments, including
  high-dimensional, singularly perturbed problems and those posed over non-convex domains.
\end{abstract}
\maketitle
\input{Bulk}
\bibliographystyle{siamplain}
\bibliography{nitsche}
\end{document}

%% file: Bulk.tex
\section{Introduction}

The numerical approximation of partial differential equations (PDEs)
using artificial neural networks (ANNs) has gained significant
attention in recent years \cite{EYu:2018, SirignanoSpiliopoulos:2018,
  RaissiPerdikarisKarniadakis:2019, Yulei_Liao_2021,
  GeorgoulisLoulakisTsiourvas:2023}. This surge is largely attributed
to the success of deep learning in various complex tasks
\cite{LeCunBengioHinton:2015, GoodfellowBengioCourville:2016}. In the
context of solving PDEs, neural network-based methods such as the deep
Ritz approach \cite{EYu:2018}, which approximates solutions by
minimising the Dirichlet energy, and the physics-informed neural
networks (PINNs) \cite{RaissiPerdikarisKarniadakis:2019}, which
minimise the $\leb{2}$-norm of the residuals are prominent examples.

Despite their success, a challenge in these methods lies in the
enforcement of boundary conditions. While classical numerical methods
also face difficulties in this regard, the non-standard nature of
neural network approximation spaces makes this issue particularly
pronounced. Standard penalty approaches often require large penalty
weights to enforce boundary conditions accurately, resulting in
ill-conditioned optimisation problems that are difficult to tune and
can lead to suboptimal solutions. This issue is particularly severe in
problems involving singular perturbations or complex
domains. Additionally, the practical implementation of boundary
conditions in ANN-based methods poses challenges, as accurately
capturing boundary data within the neural network's architecture often
proves difficult.

To address these challenges, we propose extending the Lagrange
multiplier framework from finite elements, as introduced by
Babu\v{s}ka \cite{Babuska:1973b}, to neural network-based PDE
solvers. Our work develops a  class of algorithms termed Deep
Uzawa algorithms, which iteratively solve the resulting saddle point
problems to weakly impose boundary conditions. The key innovation lies
in adapting Uzawa's algorithm \cite{Uzawa:1958} to this context,
allowing for efficient iterative approximation of PDEs where boundary
conditions are enforced using an augmented Lagrangian
formulation. This approach ensures that the algorithmic framework
remains stable and accurate due to the coercivity of the energies
involved.

The Deep Uzawa methods, RitUz and PINNUz, extend existing deep Ritz
and PINN frameworks with minimal modifications, making them highly
practical for integration into current computational workflows. The
theoretical analysis provided includes convergence proofs that
demonstrate the iterative schemes' stability at the PDE level. These
theoretical guarantees offer a robust foundation for the
implementation of the Deep Uzawa algorithms and provide insight into
their convergence behaviour. We compare the behaviour of these approaches to
the vanilla methods and show that the Deep Uzawa approach
achieves comparable or superior performance without relying on tuning penalty parameters.

Numerous methods have been explored for weakly imposing boundary
conditions within ANN-based PDE solvers. The deep Ritz method
\cite{EYu:2018} and PINNs \cite{RaissiPerdikarisKarniadakis:2019} form
the foundational basis for many current approaches, but both rely on
penalty terms for boundary enforcement, which can make optimisation
challenging, particularly when large penalties are needed. To address
these shortcomings, \cite{Yulei_Liao_2021} proposed an adaptation
using Nitsche's method \cite{Nitsche:1971} from finite element
analysis to weakly impose boundary conditions, mitigating conformity
issues highlighted in \cite{EYu:2018}. Similarly,
\cite{WangXuZhangZhang:2020} compared traditional Ritz-Galerkin
methods with ANN-based approaches, noting the implicit regularisation
properties provided by neural networks. Other advancements, such as
the penalty-free neural network strategy in \cite{ShengYang:2021},
have targeted second-order boundary value problems in complex
geometries. In high-dimensional settings, \cite{GrohsHerrmann:2022}
explored deep learning approaches for elliptic PDEs with non-trivial
boundary conditions, showcasing the flexibility of neural networks in
handling such cases.

Our Deep Uzawa method builds on these developments by leveraging a
consistent saddle point framework to address the boundary enforcement
problem, offering a minor tweak computationally that provably enhances stability and
accuracy. The application of Uzawa-type iterations in neural network
contexts, as presented in
\cite{makridakis2024deepuzawapdeconstrained}, serves as a foundation
for our iterative scheme. Our approach provides a structured way to
balance the competing objectives of PDE accuracy and boundary
condition enforcement, demonstrating improved stability in various
numerical experiments, including problems on non-convex domains and
high-dimensional geometries. 

The rest of the paper is organized as follows: In \S\ref{sec:Sobolev},
we introduce the notation and fundamental concepts related to Sobolev
spaces, which form the basis for the functional framework of our
analysis. \S\ref{sec:Dirichlet_min} presents the development of a Deep Ritz-Uzawa (RitzUz) method, including a proof of convergence in suitable
Sobolev spaces. In \S\ref{sec:PINNs_min}, we extend this approach to
the PINNs-Uzawa (PINNUz) scheme, demonstrating convergence within an
appropriate space for least-squares minimisation. The construction of
neural network approximations and their integration within the Deep
Uzawa framework are discussed in \S\ref{sec:NN}. Numerical results
showcasing the effectiveness of our methods for boundary layer
problems, those in complex geometries and high dimension are provided in
\S\ref{sec:Numresults}. 

\section{Notation and problem setup}
\label{sec:Sobolev}

We will use a standard notation for Sobolev spaces \cite{EG21-I}. For
$\W\subseteq\reals^d$, we denote by $\Norm{\cdot}_{\leb{p}(\W)}$ the
$\leb{2}(\W)$-norm with associated inner product
$\ip{\cdot}{\cdot}_{\leb{2}(\W)}$. For $s\ge 0, p\in (1,\infty)$, we
denote by $\Norm{\cdot}_{\sobh{p}(\W)}$
($\norm{\cdot}_{\sobh{p}(\W)}$) the norm (semi-norm) in
the Hilbert space $\sobh{p}(\W)$.  We will now provide a short summary of fractional
and negative Sobolev spaces and their associated norms and inner
products.

\subsection{Fractional and Negative Sobolev spaces}

For $m \in \mathbb{N}$, we define the fractional Sobolev space
$\sobh{m-\frac{1}{2}}(\partial \W)$ as the set of traces of functions
in $\sobh{m}(\W)$,
\begin{equation}
  \sobh{m-\frac{1}{2}}(\partial \W)
  :=
  \setbra{g:\partial \W \rightarrow \mathbb{R}: ~ u = g \text{ a.e. on }\partial \W, ~ \forall u \in \sobh{m}(\W)}.
\end{equation}
In general, fractional Sobolev spaces are defined for all exponents
and powers in terms of Gagliardo semi-norms. However, for the purposes
of this paper, the above definition is sufficient as we consider the
case when $m = 1,2$, which are particularly relevant for analysis of
the methods discussed.

For $s > 0$, the negative Sobolev space $\sobh{-s}(\partial \W)$ is
defined as the dual of the space $\sobh{s}(\partial \W)$ and has an
associated norm given by
\begin{equation}
  \Norm{G}_{\sobh{-s}(\partial \W)}
  :=
  \sup\limits_{v \in \sobh{s}(\partial \W)} \dfrac{\duality{G}{v}_{\sobh{-s}(\partial \W)\times \sobh{s}(\partial \W)}}{\Norm{v}_{\sobh{s}(\partial \W)}},
\end{equation}
where $\duality{G}{v}_{\sobh{-s}(\partial \W)\times \sobh{s}(\partial
  \W)}$ represents the duality pairing, i.e., the action of the linear
functional $G$ applied to the function $v$.

To facilitate later analysis, we recall the following trace theorem,
which will be used to relate functions defined on $\W$ to their
behaviour on $\partial \W$.

\begin{theorem}[{\cite[Theorem 1]{Gagliardo57}}]
\label{the:trace}
Let $s>\frac 12$. Then, for $v\in\sobh s(\W)$, there exists a constant
$C_{tr} > 0$ such that
\begin{equation}
    \Norm{v}_{\leb2(\partial \W)}
    \leq
    C_{tr} \Norm{v}_{\sobh s(\W)}.
\end{equation}
\end{theorem}

\begin{remark}[Computability of $C_{tr}$]
  \label{rem:trace}
  The trace constant $C_{tr}$ can, in many cases, be 
  estimated, particularly for standard geometries such as cubes and
  spheres \cite{cangiani2022hp}. For a domain $\W$ that is star-shaped with respect to a point $\vec x_0 \in \W$ and satisfies $(\vec x - \vec x_0) \cdot \vec n (\vec x) > 0$ for all $\vec x \in \partial \W$. Then we have that the following inequality holds
  \begin{equation}
  	C_{tr}
  	\leq 
	\dfrac{1}{2 \min\limits_{\vec x \in \partial \W}(\vec x - \vec x_0) \cdot \vec n (\vec x) }\bra{d + \sqrt{d^2 + 4 \max\limits_{\vec x \in \partial \W} \abs{\vec x - \vec x_0}^2}}
  \end{equation}
\end{remark}

\section{Dirichlet energy minimisation}\label{sec:Dirichlet_min}

In this section, we introduce the model problem and provide some
background on the Lagrange multiplier method, highlighting its
connection to the Uzawa algorithm when the Dirichlet energy defines
the loss function as in Ritz-based neural network methods.

To that end, we consider a self-adjoint elliptic problem with a strictly positive definite matrix $\vec A
\in \reals^{d\times d}$, $f \in \leb2(\W)$, and $g \in
\sobh{1/2}(\partial \W)$. To set up the problem, we define the
function space:
\begin{equation}
  \sobhg{1}(\W)
  :=
  \ensemble{\phi \in \sobh1(\W)}{\phi\vert_{\partial \W} = g}.
\end{equation}
We then seek a solution $u \in \sobhg 1(\W)$ such that:
\begin{equation}
\label{eq:elliptic}
\begin{split}
  \mathscr{L}u:=-\div\bra{\vec A \nabla u} + u &= f \text{ in } \W,
\end{split}
\end{equation}
in the weak sense, that is $u$ satisfies 
\begin{equation}
  \label{eq:elliptic-weak}
  \ip{\vec A \nabla u}{\nabla v}_{\leb2(\W)} + \ip{u}{v}_{\leb2(\W)}
  =
  \ip{f}{v}_{\leb2(\W)} \quad \Foreach v \in \sobhz1(\W).
\end{equation}
We observe that equation \eqref{eq:elliptic} is the Euler-Lagrange
equation corresponding to the minimisation of the convex quadratic
Dirichlet functional
\begin{equation}\label{eq:EL_functional}
  J_D(u)
  :=
  \frac{1}{2} \Norm{\vec A^{1/2} \nabla u}^2_{\leb2(\W)}
  + 
  \frac{1}{2}\Norm{u}^2_{\leb2(\W)}
  -
  \ip{f}{u}_{\leb2(\W)}.
\end{equation}
Furthermore, the weak solution to (\ref{eq:elliptic-weak}) also
minimises the Dirichlet energy, that is
\begin{equation}
  u = \argmin_{\phi \in \sobhg1(\W)} J_D(\phi).
\end{equation}
This variational formulation serves as the foundation for deep Ritz
neural network methods, where neural networks are utilised to
approximate the minimiser of $J_D(u)$.

\subsection{Penalty methods}

In practice, it is often beneficial to pose the minimisation problem
over a larger space, such as $\sobh1(\W)$, and enforce the boundary
condition as a constraint within the formulation. This can be
accomplished by extending the definition of
$J_{D}:\sobh1(\W)\rightarrow \reals$ through the addition of a penalty
term
\begin{equation}
  J_{D}(u)
  :=
  \frac{1}{2} \Norm{{\vec A}^{1/2} \nabla u}^2_{\leb2(\W)}
  + 
  \frac{1}{2}\Norm{u}^2_{\leb2(\W)}
  -
  \ip{f}{u}_{\leb2(\W)}
  +
  \frac{\gamma}{2}\Norm{u - g}^2_{X},
\end{equation}
for some $\gamma \geq 0$ and an appropriate normed space $X$. A
natural choice from an analytic point of of view for $X$ is
$\sobh{1/2}(\partial \W)$.

\begin{remark}[Finite element approaches]
  In the context of finite element methods, the penalty term is often
  formulated using a mesh-weighted $\leb2$ norm. Let $h$ denote the
  finite element discretisation parameter then one can consider
  \begin{equation}
    \begin{split}
      J_{D,h}(u_h)
      :=&
      \frac{1}{2} \Norm{{\vec A}^{1/2} \nabla u_h}^2_{\leb2(\W)}
      + 
      \frac{1}{2}\Norm{u_h}^2_{\leb2(\W)}
        -
      \ip{f}{u_h}_{\leb2(\W)}
      +
      \frac{\gamma}{h}\Norm{u_h - g}^2_{\leb{2}(\partial \W)}.
    \end{split}
\end{equation}
It can then be shown that, by choosing $\gamma$ sufficiently large, we
have:
\begin{equation}
  \lim_{h\to 0 } \Norm{u - u_h}_{\sobh1(\W)} = 0.
\end{equation}
This argument relies on applying an inverse estimate, however, such estimates are not available in the context of neural network approximations.
\end{remark}

\begin{definition}[Dirichlet Loss Functional]
  To maintain compatibility with neural network-based methodologies,
  we use $X = \leb2(\partial \W)$-penalty term, i.e.,
  \begin{equation}\label{eq:DG_Cost_functional}
    J_{D}(u)
    :=
    \dfrac{1}{2}\Norm{{\vec A}^{1/2}\nabla u}^2_{\leb2(\W)}
    +
    \dfrac{1}{2}\Norm{u}^2_{\leb2(\W)}
    - 
    \ip{f}{u}_{\leb2(\W)}
    +
    \dfrac{\gamma}{2}\Norm{u - g}^2_{\leb2(\partial \W)}.
  \end{equation}
\end{definition}

\begin{remark}[Limitations of penalty only approaches]
  In neural network approaches, the penalty term in
  (\ref{eq:DG_Cost_functional}) is often insufficient to guarantee the
  solution well approximates both the PDE and the boundary
  condition. This is as as the resulting Euler-Lagrange equations
  corresponding to \eqref{eq:DG_Cost_functional} are:
\begin{equation}
    \begin{split}
        \mathscr{L} u - f &= 0 \quad \text{in } \W, \\
        \vec n \cdot \vec A \nabla u + \gamma\bra{u - g} &= 0 \quad \text{on } \partial \W.
    \end{split}
\end{equation}
Thus, the minimiser of $J_{D}$ satisfies a Robin boundary condition
with parameter $\gamma$, rather than a true Dirichlet boundary
condition. This observation motivates the exploration of alternative
methods, such as those involving Lagrange multipliers
\cite{Babuska:1973b}.
\end{remark}

The first core idea of our approach is the introduction of the
Lagrangian.
\begin{definition}[Dirichlet energy Lagrangian]
  For a given $g \in \sobh{1/2}(\partial \W)$ and $f \in \leb2(\W)$,
  let the Lagrangian $L_D:\sobh1(\W)\times \sobh{-1/2}(\partial \W)
  \rightarrow \mathbb{R}$ be given by:
  \begin{equation}
    \label{eq:Dirchlet_energy_Lagrangian_defn}
    L_D(u,\lambda)
    :=
    J_{D}(u)
    -
    \duality{\lambda}{ u -g}_{\sobh{-1/2}(\partial \W) \times \sobh{1/2}(\partial \W)}.
  \end{equation}
  This Lagrangian formulation allows us to incorporate the boundary
  condition $u = g$ weakly by introducing a Lagrange multiplier
  \(\lambda\).
\end{definition}

We then seek the saddle points of the Lagrangian $L_D$, that satisfy
$(u^*, \lambda^*) \in \sobh{1}(\W) \times \sobh{-1/2}(\partial \W)$
such that:
\begin{equation}\label{defn:Dirchlet_energy_Lagrangian_saddle_points}
  (u^*, \lambda^*)
  =
  \argmin\limits_{u \in \sobh{1}(\W)}
  \argmax\limits_{\lambda \in \sobh{-1/2}(\partial \W)} L_D(u,\lambda).
\end{equation}
This problem leads to the Euler-Lagrange equations
\begin{equation}
  \begin{split}
    \mathscr{L}u^* - f &= 0 \quad \text{in } \W \\
    u^* - g &= 0 \quad \text{on } \partial \W \\
    -\lambda^* + \vec n \cdot 
    \vec A \nabla u^* &= 0 \quad \text{on } \partial \W.
  \end{split}
\end{equation}
To solve the saddle point problem
\eqref{defn:Dirchlet_energy_Lagrangian_saddle_points}, we propose an
iterative method through an Uzawa algorithm at the continuum level.

\begin{definition}[Dirichlet energy update scheme]
  \label{defn:Dirchlet_energy_Lagrangian_update_scheme}
  For a given initial guess $\lambda^0 \in \sobh{-1/2}(\partial \W)$
  and a step size $\rho > 0$, we define the sequence of functions
  $\setbra{\lambda^k}_{k=0}^\infty \subset \sobh{-1/2}(\partial \W)$
  and $\setbra{u^k}_{k=0}^\infty \subset \sobh1(\W)$ by the following
  iterative scheme:
  \begin{equation}\label{eq:Dirchlet_energy_Lagrangian_update_scheme}
    \begin{split}
      u^k &= \argmin\limits_{u \in \sobh1(\Omega)} L_D(u, \lambda^k) \\
      \duality{\lambda^{k+1} - \lambda^k}{\phi}_{\sobh{-1/2}(\partial \W) \times \sobh{1/2}(\partial \W)}
      &= - \rho \ip{u^k - g}{\phi}_{\leb2(\partial \Omega)} \Foreach \phi \in\sobh{1/2}(\partial \W). 
        \end{split}
    \end{equation}
\end{definition}

\begin{remark}[Simplifications in the Neural Network Framework]
  Our algorithm relies on approximating the Uzawa iterates
  \eqref{eq:Dirchlet_energy_Lagrangian_update_scheme} within an
  appropriate neural network framework. Due to the inherent smoothness
  of neural network functions, their traces belong to
  $L^2(\partial \W)$. Similarly, the discrete Lagrange multipliers
  $\lambda^k_\ell$ are also functions in $L^2(\partial \W)$. As a
  result, duality pairings simplify to basic $L^2(\partial \W)$
  integrals
  \begin{equation}
    \duality{\lambda_\ell }{ u_\ell  -g}_{\sobh{-1/2}(\partial \W) \times \sobh{1/2}(\partial \W)} 
    = \int _ {\partial \W} ( u_\ell  -g ) \, \lambda_\ell  \, ds.
  \end{equation}
  Additionally, updating the Lagrange multiplier in the second
  equation of \eqref{eq:Dirchlet_energy_Lagrangian_update_scheme}
  simplifies to a straightforward function evaluation on $\partial
  \W$. For more details, see Section 5.3. These observations result in
  a particularly simple and efficient implementation.
\end{remark}

\begin{theorem}\label{thm:Dirichlet_energy_convergence}
  Let the sequences $\{u^k\}$ and $\{\lambda^k\}$ denote the Uzawa
  iterates given in Definition
  \ref{defn:Dirchlet_energy_Lagrangian_update_scheme}, and
  \(\bra{u^*,\lambda^*}\) denote the saddle point of
  \eqref{defn:Dirchlet_energy_Lagrangian_saddle_points}. Let
  $C_{tr}>0$ is the trace constant associated with $\W$ and assume
  that $\vec A$ is a strictly positive definite matrix with the
  smallest eigenvalue $\sigma_{\min} > 0$. Suppose further that the
  parameters $\gamma \geq 0$ and $\rho > 0$ satisfy
  \begin{equation}
    \label{eq:Dirichlet_energy_rho_assumption}
    \rho - 2\gamma < \dfrac{2\min\setbra{\sigma_{\min},1}}{C_{\rm{tr}}^2}.
  \end{equation}
Then we have
\begin{equation}
  u^k \rightarrow u^*, \text{ in }\sobh1(\W) \text{ as } k \rightarrow \infty.
\end{equation}
\end{theorem}

\begin{remark}[Convergence Regimes]
  There are two regimes in which the inequality in equation
  \eqref{eq:Dirichlet_energy_rho_assumption} holds. The first is when
  $\gamma$ is large relative to $\rho$, corresponding to $2\gamma \geq
  \rho$. In this regime, the inequality is trivially satisfied, and
  the bound does not depend on the matrix $\vec A$. The second is when
  $\gamma$ is small relative to $\rho$, including the case when
  $\gamma = 0$, corresponding to $2\gamma < \rho$. Here, the Uzawa
  update step-size $\rho$ must be sufficiently small. These regimes
  illustrate how the balance between $\gamma$ and $\rho$ impacts
  convergence.
\end{remark}

To prove Theorem \ref{thm:Dirichlet_energy_convergence}, we will prove
or state a series of definitions and technical lemmata that will be
used in the proof. To begin, we recall the classical Riesz
Representation Theorem and some of its consequences, which play a
important role in the convergence analysis.

\begin{theorem}[Riesz Representation]
  For every $s \in \mathbb{N}$ and $z \in \sobh{-s/2}(\partial\W)$,
  there exists a unique element $R_s[z] \in \sobh{s/2}(\partial \W)$,
  known as the Riesz representor of $z$, such that for all $\phi \in
  \sobh{s/2}(\partial\W)$, the following holds:
  \begin{equation}
    \label{eq:Riesz_representor}
    \ip{R_s[z]}{\phi}_{\sobh{s/2}(\partial\W)}
    :=
    \duality{z}{\phi}_{\sobh{-s/2}(\partial \W) \times \sobh{s/2}(\partial \W)}
    \Foreach \phi \in \sobh{s/2}(\partial\W).
  \end{equation}
  Moreover, the mapping $R_s: \sobh{-s/2}(\partial\W) \to
  \sobh{s/2}(\partial\W)$ is an isometric isomorphism, i.e., it
  preserves the norm structure:
  \begin{equation}
    \Norm{R_s[z]}_{\sobh{s/2}(\partial\W)}   
    =
    \Norm{z}_{\sobh{-s/2}(\partial\W)}.
  \end{equation}
  Furthermore, as a consequence of this definition, for any $z \in \sobh{-s/2}(\partial\W)$, we have:
  \begin{equation}
    \Norm{R_s[z]}^2_{\sobh{s/2}(\partial\W)}   
    =
    \duality{z}{R_s[z]}_{\sobh{-s/2}(\partial \W) \times \sobh{s/2}(\partial \W)}.
  \end{equation}
\end{theorem}

To prove Theorem \ref{thm:Dirichlet_energy_convergence}, we first establish the following lemma, which provides bounds on the Uzawa iterates \(u^k\) and \(\lambda^k\) in relation to the saddle points \(u^*\) and \(\lambda^*\).

\begin{lemma}[Bounds on the Dirichlet Lagrangian]
  \label{lemma:bounds_on_the_dirichlet_lagrangian}
  Let the sequences $\{u^k\}, \{\lambda^k\}$ denote the Uzawa iterates
  given in Definition
  \ref{defn:Dirchlet_energy_Lagrangian_update_scheme} and
  $u^*,\lambda^*$ denote the saddle points of $L_D$
  (\ref{defn:Dirchlet_energy_Lagrangian_saddle_points}). Then we have
  \begin{equation}
    \begin{split}
      \Norm{{\vec A}^{1/2}\nabla(u^k-u^*)}^2_{\leb2(\W)}
      +
      \Norm{u^k-u^*}^2_{\leb2(\W)}
      +
      \gamma \Norm{u^k-u^*}^2_{\leb2(\partial\W)}
      \\
      =
      \duality{\lambda^k - \lambda^*}{u^k-u^*}_{\sobh{-1/2}(\partial \W) \times \sobh{1/2}(\partial \W)}.
      \end{split}
  \end{equation}
\end{lemma}

\begin{proof}
  Since $u^*$ and $\lambda^*$ are the saddle points of the Lagrangian
  $L_D$ and $\lambda^k, u^k$ satisfy Definition
  \ref{defn:Dirchlet_energy_Lagrangian_update_scheme}, we infer the
  first-order optimality conditions
  \begin{equation}
      \begin{split}
          \ip{\vec A \nabla u^*}{\nabla \phi}_{\leb2(\W)} 
          &+ \ip{u^*}{\phi}_{\leb2(\W)}
          -\ip{f}{\phi}_{\leb2(\W)} 
          \\
          &+ \gamma \ip{u^*-g}{\phi}_{\leb2(\partial\W)}
          - \duality{\lambda^*}{\phi}_{\sobh{-1/2}(\partial \W) \times \sobh{1/2}(\partial \W)} = 0, 
          \\
          \ip{\vec A \nabla u^k}{\nabla \phi}_{\leb2(\W)} 
          &+ \ip{u^k}{\phi}_{\leb2(\W)}
          -\ip{f}{\phi}_{\leb2(\W)}
          \\
          &+ \gamma \ip{u^k-g}{\phi}_{\leb{2}(\partial\W)}
          - \duality{\lambda^k}{\phi}_{\sobh{-1/2}(\partial \W) \times \sobh{1/2}(\partial \W)} = 0, 
      \end{split}
  \end{equation}
  for all $\phi \in \sobh{1}(\W)$.

  Taking the difference of the above expressions and setting $\phi =
  u^k - u^*$ gives the following error equation
  \begin{equation}
    \begin{split}
      \Norm{{\vec A}^{1/2}\nabla(u^k-u^*)}^2_{\leb2(\W)} 
      +
      \Norm{u^k-u^*}^2_{\leb2(\W)} 
      +
      \gamma \Norm{u^k-u^*}^2_{\leb2(\partial\W)}
      \\
      = 
      \duality{\lambda^k - \lambda^*}{u^k-u^*}_{\sobh{-1/2}(\partial \W) \times \sobh{1/2}(\partial \W)},
    \end{split}
  \end{equation}
  concluding the proof.
\end{proof}

\begin{lemma}\label{lemma:Dirichlet_energy_saddle_point_property}
  Let $\{u^k\}, \{\lambda^k\}$ be given by
  (\ref{eq:Dirchlet_energy_Lagrangian_update_scheme}), and $u^*,
  \lambda^*$ denote the saddle points of
  (\ref{defn:Dirchlet_energy_Lagrangian_saddle_points}). Then,
  \begin{equation}
    \begin{split}
      \Norm{\lambda^{k+1}-\lambda^{*}}^2_{\sobh{-1/2}(\partial\W)}
      &=
      \Norm{\lambda^k-\lambda^{*}}^2_{\sobh{-1/2}(\partial\W)} 
      +
      \rho^2 \Norm{u^k - u^*}^2_{\leb{2}(\partial\W)}.
      \\
      &\qquad
      - 2\rho \duality{\lambda^{k}-\lambda^{*}}{u^k - u^*}_{\sobh{-1/2}(\partial \W) \times \sobh{1/2}(\partial \W)}.
    \end{split}
  \end{equation}
\end{lemma}
\begin{proof}
  We begin by recalling the first-order optimality conditions for the
  saddle points and iterates of $L_D$. By these conditions, we have
  \begin{equation}
    \label{eq:updateustar}
    \begin{split}
      \duality{\lambda^{*}}{\psi}_{\sobh{-1/2}(\partial \W) \times \sobh{1/2}(\partial \W)}
      =& \duality{\lambda^*}{\psi}_{\sobh{-1/2}(\partial \W) \times \sobh{1/2}(\partial \W)}
      \\
      &- \rho \ip{u^* - g}{\psi}_{\leb2(\partial \W)} \Foreach \psi \in \sobh{1/2}(\partial\W).
    \end{split}
  \end{equation}  
  Taking the difference between equation (\ref{eq:updateustar}) and
  the update condition in
  (\ref{eq:Dirchlet_energy_Lagrangian_update_scheme}), we obtain
  \begin{equation}\label{eq:Lemma3.3pt0}
    \begin{split}
      \duality{\lambda^{k+1}-\lambda^{*}}{\psi}_{\sobh{-1/2}(\partial \W) \times \sobh{1/2}(\partial \W)}
      =& \duality{\lambda^k-\lambda^{*}}{\psi}_{\sobh{-1/2}(\partial \W) \times \sobh{1/2}(\partial \W)}
      \\
      &- \rho \ip{u^k - u^*}{\psi}_{\leb2(\partial \W)} \Foreach \psi \in \sobh{1/2}(\partial\W).
    \end{split}
  \end{equation}
  Now, let $\psi = R_1[\lambda^{k+1}-\lambda^{*}]$. Since $R_1$ is an
  isometric isomorphism, we have
  \begin{equation}\label{eq:Lemma3.3pt1}
    \begin{split}
      \Norm{\lambda^{k+1}-\lambda^{*}}^2_{\sobh{-1/2}(\partial\W)}
      =& \duality{\lambda^k-\lambda^{*}}{R_1[\lambda^{k+1}-\lambda^{*}]}_{\sobh{-1/2}(\partial \W) \times \sobh{1/2}(\partial \W)}
      \\
      &- \rho \ip{u^k - u^*}{R_1[\lambda^{k+1}-\lambda^{*}]}_{\leb2(\partial \W)}.
    \end{split}
  \end{equation}
  The duality pairing is symmetric with respect to the Riesz
  representor
  \begin{equation}
    \duality{w}{R_1[z]}_{\sobh{-1/2}(\partial \W) \times \sobh{1/2}(\partial \W)} = \duality{z}{R_1[w]}_{\sobh{-1/2}(\partial \W) \times \sobh{1/2}(\partial \W)} \Foreach w,z \in \sobh{-1/2}(\partial\W).
  \end{equation}
  Taking $\psi = R_1[\lambda^{k}-\lambda^{*}]$ in equation \eqref{eq:Lemma3.3pt0} and substituting into \eqref{eq:Lemma3.3pt1} gives:
  \begin{equation}\label{eq:Lemma3.3pt2}
    \begin{split}
      \Norm{\lambda^{k+1}-\lambda^{*}}^2_{\sobh{-1/2}(\partial\W)}
      =& \Norm{\lambda^k-\lambda^{*}}_{\sobh{1/2}(\partial \Omega)}^2
      - \rho \ip{u^k - u^*}{R_1[\lambda^{k}-\lambda^{*}]}_{\leb2(\partial \Omega)}
      \\
      &- \rho \ip{u^k - u^*}{R_1[\lambda^{k+1}-\lambda^{*}]}_{\leb{2}(\partial \W)}.
    \end{split}
  \end{equation}
  Finally, taking $\psi = u^k - u^*$ in equation
  \eqref{eq:Lemma3.3pt0} and substituting into \eqref{eq:Lemma3.3pt2},
  we have
  \begin{equation}
    \begin{split}
      \Norm{\lambda^{k+1}-\lambda^{*}}^2_{\sobh{-1/2}(\partial\W)}
      =& \Norm{\lambda^k-\lambda^{*}}_{\sobh{1/2}(\partial \Omega)}^2
      - 2\rho \ip{u^k - u^*}{R_1[\lambda^k-\lambda^{*}]}_{\leb2(\partial \Omega)}
      \\
      &+ \rho^2 \Norm{u^k - u^*}^2_{\leb2(\partial \Omega)},
    \end{split}
  \end{equation}
  completing the proof.
\end{proof}

\subsection{Proof of Theorem \ref{thm:Dirichlet_energy_convergence}}
Let $e^k := u^k - u^*$ and $\beta := 2 \gamma - \rho$. By applying
Lemma \ref{lemma:bounds_on_the_dirichlet_lagrangian} to Lemma
\ref{lemma:Dirichlet_energy_saddle_point_property}, we obtain
\begin{equation}
  \begin{split}
    \Norm{\lambda^{k+1}-\lambda^*}_{\sobh{-1/2}(\partial \W)}^2
    &=
    \Norm{\lambda^{k}-\lambda^*}_{\sobh{-1/2}(\partial \W)}^2
    -
    2 \rho \Norm{{\vec A}^{1/2}\nabla e^k}^2_{\leb2(\W)}
    -
    2 \rho \Norm{e^k}^2_{\leb2(\W)}
    \\
    &\qquad
    -
    \rho \beta \Norm{e^k}^2_{\leb{2}(\partial\W)}.
  \end{split}
\end{equation}
Without loss of generality, we assume $e^k \neq 0$.

\subsubsection{Case 1: \(\beta \geq 0\) (Non-negative penalty term)}
If $\beta \geq 0$, then
\begin{equation}
  \begin{split}
    \Norm{\lambda^{k+1}-\lambda^*}_{\sobh{-1/2}(\partial \W)}^2
    &\leq
    \Norm{\lambda^{k}-\lambda^*}_{\sobh{-1/2}(\partial \W)}^2
    -
    2 \rho \Norm{{\vec A}^{1/2}\nabla e^k}^2_{\leb2(\W)}
    -
    2 \rho \Norm{e^k}^2_{\leb2(\W)}
    \\
    &\leq
    \Norm{\lambda^{k}-\lambda^*}_{\sobh{-1/2}(\partial \W)}^2
    -
    2 \rho \min\setbra{\sigma_{\min},1} \Norm{e^k}^2_{\sobh{1}(\Omega)}.
  \end{split}
\end{equation}
Assuming $\rho > 0$ and using the positivity of $\vec A$, we conclude
that $\Norm{\lambda^{k}-\lambda^*}_{\sobh{-1/2}(\partial \W)}$ is a
strictly decreasing sequence. Thus,
\begin{equation}
  \begin{split}
    0 \leq \lim\limits_{k \rightarrow \infty}\Norm{e^k}_{\sobh1(\W)}^2 
    &
    \leq 
    \lim\limits_{k \rightarrow \infty}\dfrac{\Norm{\lambda^{k}-\lambda^*}_{\sobh{-1/2}(\partial \W)}^2 - \Norm{\lambda^{k+1}-\lambda^*}_{\sobh{-1/2}(\partial \W)}^2}{2 \rho \min\setbra{\sigma_{\min},1}} = 0.
  \end{split}
\end{equation}

\subsubsection{Case 2: \(\beta < 0\) (Negative penalty term)}
If $\beta < 0$, then by Theorem \ref{the:trace}:
\begin{equation}
  \begin{split}
    \Norm{\lambda^{k+1}-\lambda^*}_{\sobh{-1/2}(\partial \W)}^2 
    &=
    \Norm{\lambda^{k}-\lambda^*}_{\sobh{-1/2}(\partial \W)}^2 
    -
    2 \rho \Norm{\vec A^{1/2}\nabla e^k}^2_{\leb2(\W)} 
    \\
    &\qquad 
    -
    2 \rho \Norm{e^k}^2_{\leb2(\W)}
    +
    \rho |\beta| \Norm{e^k}_{\leb{2}(\partial \W)}^2
    \\
    &\leq 
    \Norm{\lambda^{k}-\lambda^*}_{\sobh{-1/2}(\partial \W)}^2 
    -
    2 \rho \Norm{\vec A^{1/2}\nabla e^k}^2_{\leb2(\W)} 
    \\
    &\qquad 
    -
    2 \rho \Norm{e^k}^2_{\leb2(\W)}
    +
    \rho |\beta| C_{\rm{tr}}^2 \Norm{e^k}_{\sobh{1}(\W)}^2.
  \end{split}
\end{equation}
Given that $\vec A$ is positive definite with the smallest eigenvalue
$\sigma_{\min} > 0$, we obtain
\begin{equation}\label{eq:Dirichlet_rate_of_conv}
  \begin{split}
    \Norm{\lambda^{k+1}-\lambda^*}_{\sobh{-1/2}(\partial \W)}^2 
    &\leq  
    \Norm{\lambda^{k}-\lambda^*}_{\sobh{-1/2}(\partial \W)}^2 
    - 2 \rho \sigma_{\min}\Norm{\nabla e^k}^2_{\leb2(\W)} 
    \\
    &\quad
    - 2 \rho \Norm{e^k}^2_{\leb2(\W)} 
    + \rho |\beta| C_{\rm{tr}}^2 \Norm{e^k}_{\sobh{1}(\W)}^2
    \\
    &\leq 
    \Norm{\lambda^{k}-\lambda^*}_{\sobh{-1/2}(\partial \W)}^2 
    - \alpha \Norm{e^k}_{\sobh1(\W)}^2, 
    \\
    \alpha &:= \rho  \bra{2\min\setbra{\sigma_{\min},1} - |\beta| C_{\rm{tr}}^2}.
  \end{split}
\end{equation}
By equation \eqref{eq:Dirichlet_energy_rho_assumption},
$\Norm{\lambda^{k}-\lambda^*}_{\sobh{-1/2}(\partial \W)}$ is a
strictly decreasing sequence. Hence,
\begin{equation}
  \begin{split}
    \lim\limits_{k \rightarrow \infty}\Norm{e^k}_{\sobh1(\W)}^2 
    &
    \leq \dfrac{1}{\alpha}
    \lim\limits_{k \rightarrow \infty} \Norm{\lambda^{k}-\lambda^*}_{\sobh{-1/2}(\partial \W)}^2 - \Norm{\lambda^{k+1}-\lambda^*}_{\sobh{-1/2}(\partial \W)}^2 \rightarrow 0,
  \end{split}
\end{equation}
completing the proof of Theorem
\ref{thm:Dirichlet_energy_convergence}. \qed

\section{Least squares minimisation}\label{sec:PINNs_min}

In the previous section, we developed an iterative scheme based on the
Euler-Lagrange equation of the cost functional $J_{D}$. However, it is
important to note that $J_{D}$ is not the only choice of cost
functional. An alternative approach is to consider a least squares
minimisation of the residual of the PDE. This method forms the
foundation of many physics-informed neural network (PINN) algorithms.

Let us assume $\W$ is a convex domain and consider a more regular
boundary function $g \in \sobh{3/2}(\partial \W)$, defining the cost
functional
\begin{equation}\label{eq:PINNs_Cost_functional_gamma=0}
  J_{R}(u) := \dfrac 12\Norm{\mathscr{L}u - f}^2_{\leb2(\W)},
\end{equation}
over the constrained space:
\begin{equation}
  \sobhg 2(\W) := \{ u \in \sobh 2(\W) : u|_{\partial \W} = g\}.
\end{equation}
This space ensures that the solution adheres to the boundary
conditions and is suitable for PINN implementations. However, similar
to the Dirichlet energy minimisation approach often it is practical to
extend the functional to include a penalty term for $\gamma \geq 0$
\begin{equation}\label{eq:PINNs_Cost_functional}
  J_{R}(u) := \dfrac 12\Norm{\mathscr{L}u - f}^2_{\leb2(\W)}
  +
  \dfrac{\gamma}{2}\Norm{u - g}^2_{\leb2(\partial \W)}.
\end{equation}
A drawback of this approach is that the boundary penalty $\Norm{u -
  g}^2_{\leb2(\W)}$ is weak, making the resulting loss $J_{R}(u)$
unbalanced. Consequently, boundary errors often dominate PINN
algorithm inaccuracies, especially in singularly perturbed
problems. To address this, we consider the critical points of an
equivalent functional over $\sobh{2}(\W)$.

\begin{remark}[Non-convex domains]
  \label{rem:nonconvex}
  The analysis presented here assumes $\W$ is convex to leverage full
  elliptic regularity of the critical points. Extending this analysis
  to non-convex domains introduces geometric singularities, which can
  be addressed by considering the space:
  \begin{equation}
    \sobhg{\mathscr L}(\W)
    :=
    \{ u \in \sobhg{1}(\W) : \mathscr{L} u \in \leb{2}(\W) \}.
  \end{equation}
  However, to avoid excessive notation and complexity, we do not
  present this analysis here. For illustration, we
  demonstrate the method applied to a non-convex domain in \S
  \ref{ex:Example_L_shaped}.
\end{remark}

\begin{definition}[PINNs energy Lagrangian]
  For a given $g \in \sobh{3/2}(\partial \W)$ and $f \in \leb2(\W)$,
  let the Lagrangian $L_{R}:\sobh{2}(\W) \times \sobh{-3/2}(\partial
  \W) \rightarrow \mathbb{R}$ be defined by
  \begin{equation}\label{eq:PINNs_energy_Lagrangian_defn_no2}
    L_{R}(u,\lambda)
    :=
    J_{R}(u) - \duality{\lambda}{u - g}_{\sobh{-3/2}(\partial \W) \times \sobh{3/2}(\partial \W)}.
    \end{equation}
\end{definition}
We seek the saddle points of the Lagrangian $L_{R}$, denoted by $(u^*,
\lambda^*) \in \sobh{2}(\W) \times \sobh{-3/2}(\partial \W)$, such
that:
\begin{equation}\label{defn:PINNs_energy_Lagrangian_saddle_points_no2}
  (u^*, \lambda^*)
  =
  \argmin\limits_{u \in \sobh{2}(\W)}\argmax\limits_{\lambda \in \sobh{-3/2}(\partial \W)}~ L_{R}(u,\lambda).
\end{equation}
This can be achieved through an Uzawa iteration scheme.

\begin{definition}[PINNs energy update scheme]\label{defn:PINNs_energy_Lagrangian_update_scheme_no2}
  For a given initial guess $\lambda_0 \in \sobh{-3/2}(\partial \W)$,
  a step size $\rho > 0$ and a penalty parameter $\gamma > 0$, we
  define the sequences $\setbra{\lambda^k}_{k=1}^\infty \subset
  \sobh{-3/2}(\partial \W)$ and $\setbra{u^k}_{k=1}^\infty \subset
  \sobh2(\W)$ by the following iterative scheme:
    \begin{equation}\label{eq:PINNs_energy_Lagrangian_update_scheme_no2}
      \begin{split}
        u^k &= \argmin\limits_{u \in \sobh2(\Omega)} L_{R}(u, \lambda^k), \\
        \duality{\lambda^{k+1}-\lambda^k}{\psi}_{\sobh{-3/2}(\partial \W) \times \sobh{3/2}(\partial \W)}
        &= 
        - \rho \ip{u^k - g}{\psi}_{\leb2(\partial \W)}, \qquad \forall \psi \in \sobh{3/2}(\partial \W).
      \end{split}
    \end{equation}
\end{definition}
As in the case of the Dirichlet energy approach, our algorithm relies
on approximating \eqref{eq:PINNs_energy_Lagrangian_update_scheme_no2}
within an appropriate neural network framework. Due to the inherent
smoothness of neural network functions, the duality pairings again
simplify to $L^2(\partial \W)$ integrals
\begin{equation}
  \duality{\lambda_\ell}{u_\ell - g}_{\sobh{-3/2}(\partial \W) \times \sobh{3/2}(\partial \W)}
  = \int _{\partial \W} (u_\ell - g) \, \lambda_\ell \, dS.
\end{equation}
Furthermore, updating the Lagrange multiplier in the second equation
of \eqref{eq:PINNs_energy_Lagrangian_update_scheme_no2} simplifies to
a function evaluation on $\partial \W$. For more details, see Section
5.4.

To distinguish this method from the previous section, we shall refer
to the Deep Ritz Uzawa scheme as RitUz and the PINNs Uzawa scheme as
PINNUz. This terminology helps to clearly identify the different
iterative schemes discussed. From this scheme, we can demonstrate the
convergence of the sequence $u^k$ in an appropriate norm.

\begin{theorem}\label{thm:PINNs_convergence}
  Let the sequence $u^k, \lambda^k$ denote the Uzawa iterates from
  Definition \ref{defn:PINNs_energy_Lagrangian_update_scheme_no2}, and
  let $u^*, \lambda^*$ denote the saddle points
  (\ref{defn:PINNs_energy_Lagrangian_saddle_points_no2}). Assume that
  $\vec A$ is a strictly positive definite matrix, and that the Uzawa
  constant $\rho > 0$ and the stabilisation constant $\gamma$ satisfy
  the following inequality:
    \begin{equation}\label{eq:PINNs_energy_gamma_assumption_no2}
        2 \gamma - \rho > 2.
    \end{equation}
    Then, for all $s \in [0, 1/2)$, we have
      \begin{equation}
        u^k \rightarrow u^*, \text{ in } \sobh{s}(\W) \text{ and } \leb2(\partial \W) \text{ as } k \rightarrow \infty.
      \end{equation}
\end{theorem}

\begin{remark}[Weaker convergence for least squares]
  Comparing Theorems \ref{thm:Dirichlet_energy_convergence} and
  \ref{thm:PINNs_convergence}, we observe several key differences.

  Firstly, in the PINNUz scheme, $u^k$ converges in $\sobh{s}$ for $s
  \in [0, 1/2)$, whereas in the RitUz scheme, $u^k$ converges in
    $\sobh{1}$. This difference arises because the Dirichlet energy
    functional is coercive over all of $\sobh{1}(\W)$, which
    facilitates stronger convergence properties in the RitUz
    scheme. On the other hand, the least squares energy functional
    used in the PINNUz scheme is not coercive over $\sobh{2}(\W)$;
    instead, it only exhibits coercivity over a weaker space. This
    reduced coercivity leads to the need for weaker regularity
    conditions for convergence in the PINNUz scheme.

    Secondly, the inequality in
    \eqref{eq:PINNs_energy_gamma_assumption_no2} does not depend on
    the matrix $\vec A$, unlike the corresponding inequality in the
    RitUz scheme \eqref{eq:Dirichlet_energy_rho_assumption}. This
    difference can be understood by recalling the two regimes in the
    RitUz scheme: when $\beta \leq 0$ and when $\beta > 0$. Dependence
    on the matrix $\vec A$ was only required in the proof for the case
    $\beta \leq 0$. In contrast, the PINNUz scheme operates solely in
    a regime where $\beta > 0$, corresponding to sufficiently large
    $\gamma$. Specifically, $\gamma$ must be strictly greater than 1
    in the PINNUz scheme, whereas the RitUz scheme could accommodate
    cases where $\gamma = 0$.    
\end{remark}

Similar to the proof of Theorem
\ref{thm:Dirichlet_energy_convergence}, we will begin by stating a
series of lemmata that will be used in the proof. Some of these
lemmata are analogous to those in the RitUz scheme but are adapted to
highlight the main differences in the PINNUz scheme.
\begin{lemma}[Bounds on the PINNs Lagrangian]\label{lemma:bounds_on_the_PINNs_lagrangian_no2}
  Let $u^k, \lambda^k$ be as in Definition \ref{defn:PINNs_energy_Lagrangian_update_scheme_no2}, and $u^*, \lambda^*$ be as in Definition \ref{defn:PINNs_energy_Lagrangian_saddle_points_no2}. Then:
  \begin{equation}
    \begin{split}
      \Norm{\mathscr{L}(u^k - u^*)}^2_{\leb2(\W)} + \gamma \Norm{u^k - u^*}^2_{\leb2(\partial \W)}     
      = \duality{\lambda^k - \lambda^*}{u^k - u^*}_{\sobh{-3/2}(\partial \W) \times \sobh{3/2}(\partial \W)}.
    \end{split}
  \end{equation}
\end{lemma}

\begin{proof}
  Similar to the proof of Lemma
  \ref{lemma:bounds_on_the_dirichlet_lagrangian}, we compute the
  Euler-Lagrange equations for $u^*$ and $u^k$ to deduce:
    \begin{equation}
        \begin{split}
            \ip{\mathscr{L}u^* - f}{\mathscr{L}\phi}_{\leb2(\W)} + \gamma \ip{u^* - g}{\phi}_{\leb2(\partial \W)} 
            &= \duality{\lambda^*}{u^* - g}_{\sobh{-3/2}(\partial \W) \times \sobh{3/2}(\partial \W)},
            \\
            \ip{\mathscr{L}u^k - f}{\mathscr{L}\phi}_{\leb2(\W)} + \gamma \ip{u^k - g}{\phi}_{\leb2(\partial \W)} 
            &= \duality{\lambda^k}{u^k - g}_{\sobh{-3/2}(\partial \W) \times \sobh{3/2}(\partial \W)}.
        \end{split}
    \end{equation}
    Considering the difference between these two equalities and
    setting $\phi = u^k - u^*$ yields the desired result.
\end{proof}

\begin{lemma}\label{lemma:PINNs_energy_saddle_point_property_no2}
  Let $\{u^k, \lambda^k\}$ be as given in
  (\ref{eq:PINNs_energy_Lagrangian_update_scheme_no2}), and $\{u^*,
  \lambda^*\}$ be as given in
  (\ref{defn:PINNs_energy_Lagrangian_saddle_points_no2}). Then
  \begin{equation}
    \begin{split}
      \Norm{\lambda^{k+1} - \lambda^*}^2_{\sobh{-3/2}(\partial\W)}
      &= \Norm{\lambda^k - \lambda^*}^2_{\sobh{-3/2}(\partial\W)}
      \\
      &\quad - 2\rho \duality{\lambda^k - \lambda^*}{u^k - u^*}_{\sobh{-3/2}(\partial \W) \times \sobh{3/2}(\partial \W)}
      \\
      &\quad + \rho^2 \Norm{u^k - u^*}^2_{\leb2(\partial \W)}.
    \end{split}
  \end{equation}
\end{lemma}

\begin{proof}
  The proof follows a similar structure to Lemma
  \ref{lemma:Dirichlet_energy_saddle_point_property}, but instead
  utilises the representor $R_3: \sobh{-3/2}(\partial \W) \rightarrow
  \sobh{3/2}(\partial \W)$ in place of $R_1$.
\end{proof}

\begin{lemma}\label{Lemma:EllipticCoercivityEstimate}
  Let $\{u^k, \lambda^k\}$ be given by
  (\ref{eq:PINNs_energy_Lagrangian_update_scheme_no2}), and $\{u^*,
  \lambda^*\}$ be given by
  (\ref{defn:PINNs_energy_Lagrangian_saddle_points_no2}). Then the
  sequence of functions defined by $e^k := u^k - u^*$ weakly satisfies
  the PDE
    \begin{equation}
        \mathscr{L}e^k = 0 \text{ on } \Omega, \quad \forall k > 0,
    \end{equation}
    subject to non-trivial boundary conditions. Additionally, for all
    $s \in [0, 1/2)$, there exists a constant $C_{\rm{reg}} > 0$
      dependent on $\mathscr L$, such that:
    \begin{equation}
        \Norm{e^k}_{\sobh{s}(\W)} \leq C_{\rm{reg}} \bra{\Norm{e^k}_{\leb2(\partial \W)} + \Norm{\mathscr{L}e^k}_{\leb2(\W)}}.
    \end{equation}
\end{lemma}

\begin{proof}
    The initial analysis is similar to that found in \cite{gazoulis2023stabilityconvergencephysicsinformed}. The functions $u^k$ and $u^*$ satisfy the following Euler-Lagrange equations:
    \begin{equation}
        \begin{split}
            \ip{\mathscr{L}u^* - f}{\mathscr{L}\phi}_{\leb2(\W)} + \gamma \ip{u^* - g}{\phi}_{\leb2(\partial \W)} &= \duality{\lambda^*}{\phi}_{\sobh{-3/2}(\partial \W) \times \sobh{3/2}(\partial \W)},
            \\
            \ip{\mathscr{L}u^k - f}{\mathscr{L}\phi}_{\leb2(\W)} + \gamma \ip{u^k - g}{\phi}_{\leb2(\partial \W)} &= \duality{\lambda^k}{\phi}_{\sobh{-3/2}(\partial \W) \times \sobh{3/2}(\partial \W)},
        \end{split}
    \end{equation}
    for all $\phi \in \sobh{2} \cap \sobhz{1}(\W)$. Thus, for $e^k := u^k - u^*$, we have:
    \begin{equation}
        \begin{split}
            \ip{\mathscr{L}e^k}{\mathscr{L}\phi}_{\leb2(\W)} + \gamma \ip{e^k}{\phi}_{\leb2(\partial \W)} = \duality{\lambda^k - \lambda^*}{\phi}_{\sobh{-3/2}(\partial \W) \times \sobh{3/2}(\partial \W)}.
        \end{split}
    \end{equation}

    Let $w \in \leb2(\W)$ be arbitrary, and consider $\phi$ as the solution of $\mathscr{L} \phi = w$ with zero boundary conditions. Then:
    \begin{equation}
        \begin{split}
            \ip{\mathscr{L}e^k}{\mathscr{L}\phi}_{\leb2(\W)} + \gamma \ip{e^k}{\phi}_{\leb2(\partial \W)} - \duality{\lambda^k - \lambda^*}{\phi}_{\sobh{-3/2}(\partial \W) \times \sobh{3/2}(\partial \W)} \\
            = \ip{\mathscr{L}e^k}{w}_{\leb2(\W)} = 0, \quad \forall w \in \leb2(\W).
        \end{split}
    \end{equation}

    Thus, $e^k$ weakly satisfies $\mathscr{L} e^k = 0$ on $\W$. Using the elliptic regularity estimate from \cite{Berggren2004}, we deduce the desired result.
\end{proof}

\subsection{Proof of Theorem \ref{thm:PINNs_convergence}}\label{sec:proofofPINNUz}
To prove Theorem \ref{thm:PINNs_convergence}, we use the lemmata
developed earlier and a similar approach to the proof of Theorem
\ref{thm:Dirichlet_energy_convergence}. With $\beta := 2 \gamma -
\rho$ and $e^k := u^k - u^*$, by Lemmas
\ref{lemma:bounds_on_the_PINNs_lagrangian_no2} and
\ref{lemma:PINNs_energy_saddle_point_property_no2}, we have
\begin{equation}
    \begin{split}
        \Norm{\lambda^{k+1} - \lambda^*}^2_{\sobh{-3/2}(\partial \W)}
        &= \Norm{\lambda^k - \lambda^*}^2_{\sobh{-3/2}(\partial \W)} \\
        &\quad - \rho \beta \Norm{e^k}^2_{\leb2(\partial \W)} - 2\rho \Norm{\mathscr{L}e^k}^2_{\leb2(\W)}.
    \end{split}
\end{equation}
Applying Lemma \ref{Lemma:EllipticCoercivityEstimate} along with
Young's inequality, we obtain
\begin{equation}
    \begin{split}
        \Norm{\lambda^{k+1} - \lambda^*}^2_{\sobh{-3/2}(\partial \W)} - \Norm{\lambda^k - \lambda^*}^2_{\sobh{-3/2}(\partial \W)}
        \\ \leq -\rho (\beta - 2) \Norm{e^k}^2_{\leb2(\partial \W)} 
        - \dfrac{4\rho}{C_{\rm{reg}}^2} \Norm{e^k}^2_{\sobh{s}(\W)}.
    \end{split}
\end{equation}
Without loss of generality, assuming $e^k \neq 0$, we see that
$\Norm{\lambda^k - \lambda^*}_{\sobh{-3/2}(\partial \W)}$ is a
strictly decreasing sequence. By using
\eqref{eq:PINNs_energy_gamma_assumption_no2} and the assumption that
$\rho > 0$, the rest of the proof follows similarly to the proof of
Theorem \ref{thm:Dirichlet_energy_convergence}.

\begin{remark}[Impact of Stronger Penalties on Convergence]
  By increasing the strength of the boundary penalty term to
  higher-order Sobolev norms, such as $\sobh{1}(\partial \W)$, it is
  possible to establish convergence of $u^k$ in stronger norms, such
  as $\sobh{3/2}(\W)$. Implementing such an $\sobh{1}(\W)$ penalty can
  be practically realised as it requires incorporating tangential
  derivatives, which can be efficiently computed and accounted for in
  many neural network frameworks.
\end{remark}
    
\section{Neural Networks and Deep Uzawa}\label{sec:NN}

This section outlines the methodology for constructing neural network
approximation schemes and their integration within the Deep Uzawa
framework, illustrated through examples.

\subsection{Neural Network Function Approximation}

Consider functions $u_\theta$ approximated by neural networks. A deep
neural network maps each point $x \in \W$ to a value $w_\theta(x) \in
\mathbb{R}$ through the process
\begin{equation}
  \label{eq:NN1}
  w_\theta(x) := \mathcal{C}_L \circ \sigma \circ \mathcal{C}_{L-1} \circ \cdots \circ \sigma \circ \mathcal{C}_1(x),
\end{equation}
where $\mathcal{C}_k$ represents affine transformations defined by
\begin{equation}
  \label{eq:NN2}
  \mathcal{C}_k y = W_k y + b_k, \quad W_k \in \mathbb{R}^{d_{k+1} \times d_k}, \, b_k \in \mathbb{R}^{d_{k+1}}.
\end{equation}

The parameters $\theta = \{W_k, b_k\}_{k=1}^L$ collectively define the
network $\mathcal{C}_L$. The set of all such networks is denoted by
$\mathcal{N}$, with the space of functions represented by:
\begin{equation}
  \mathcal{V}_N = \{u_\theta : \W \rightarrow \mathbb{R} \mid u_\theta(x) = \mathcal{C}_L(x) \text{ for some } \mathcal{C}_L \in \mathcal{N}\}.
\end{equation}
Note that $\mathcal{V}_N$ is not a linear space, although $\Theta =
\{\theta \mid u_\theta \in \mathcal{V}_N\}$ is a linear subspace of
$\mathbb{R}^{\dim \mathcal{N}}$.

\begin{figure}[hbtp]
  \input{nnText}
  \caption{Illustration of $\mathcal{C}_L$.}
\end{figure}
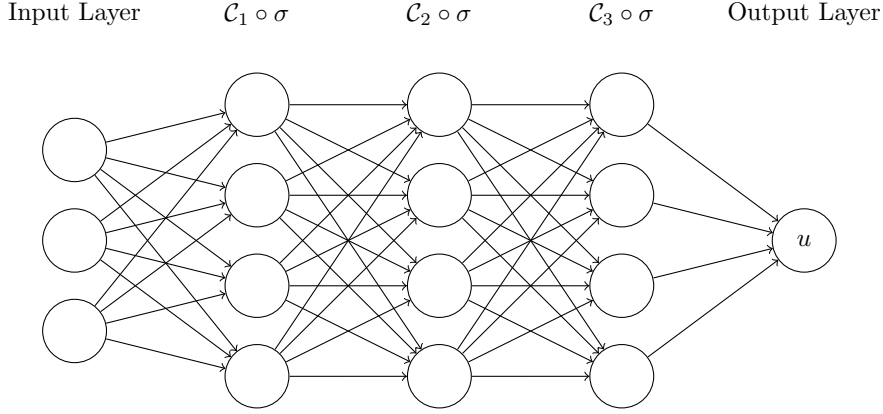

\subsection{Training within the Deep Uzawa Framework}\label{sec:Quadrature}

To implement the Deep Uzawa iteration, we require discrete versions of
the energy functionals $L_D$ and $L_R$. These can be approximated
using quadrature rules for integrals over $\W$ and $\partial \W$. Let
$\mathcal{K}_h$ and $\partial \mathcal{K}_h$ represent sets of
discrete points in $\W$ and on $\partial \W$, respectively, with
weights $w_y$ and $w_b$. We approximate the RitzUz functional as
\begin{equation}
  \sum_{y \in \mathcal{K}_h} w_y g(y) \approx \int_{\W} g(x) \, dx, \quad \sum_{b \in \partial \mathcal{K}_h} w_b g(b) \approx \int_{\partial \W} g(x) \, ds.
\end{equation}
That is
\begin{equation}
  \begin{split}
      L_D(u, \lambda) \approx& \sum_{y \in \mathcal{K}_h} w_y \left( \frac{1}{2} |\vec{A}^{1/2} \nabla u|^2(y) + \frac{1}{2} u^2(y) - u(y) f(y) \right) \\
      &+ \sum_{b \in \partial \mathcal{K}_h}
      w_b\bra{ \frac{\gamma}{2} (u - g)^2(b) 
        -
        (u(b) - g(b)) \lambda(b)
        }.
  \end{split}
\end{equation}
For the PINNUz iteration, we define the discrete functional similarly
through
\begin{equation}
  \begin{split}
      L_{R}(u, \lambda) \approx& \frac{1}{2} \sum_{y \in \mathcal{K}_h} w_y \norm{\mathscr{L}u - f}^2(y) \\
      &+ 
      \sum_{b \in \partial \mathcal{K}_h}
      w_b \bra{ \frac{\gamma}{2}(u - g)^2(b) 
        -  (u(b) - g(b)) \lambda(b) }.
  \end{split}
\end{equation}

\subsection{The Deep Uzawa Iteration}

The Deep Uzawa algorithm integrates the Uzawa iteration with neural
network-based minimisation, alternating between minimising the
discrete energy functional and updating the Lagrange multiplier. The
procedure is detailed in Algorithm \ref{alg:uzawa}.

\begin{algorithm}[h!]
  \caption{Deep Uzawa Iteration}\label{alg:uzawa}
  \begin{algorithmic}[1]
    \Require Initial guess $\lambda^0$, Uzawa step size $\rho > 0$, number of Uzawa steps $N_{\text{Uz}}$, number of SGD iterations $N_{\text{SGD}}$, learning rate $\eta$.
    \State $k \gets 0$
    \State Initialise neural network parameters $\theta^0$
    \For{$k = 1$ to $N_{\text{Uz}}$}
        \For{$m = 0$ to $N_{\text{SGD}} - 1$}
          \State Compute stochastic gradient $\nabla_\theta L_{Q,\gamma}(u_\theta^m, \lambda^k)$
          \State Update parameters: $\theta^{m+1} \gets \theta^m - \eta \nabla_\theta L_{Q,\gamma}(u_\theta^m, \lambda^k)$
        \EndFor
        \State $u^k \gets u_\theta^{N_{\text{SGD}}}$
        \State Update Lagrange multiplier: $\lambda^{k+1}(b) \gets \lambda^k(b) + \rho (u^k(b) - g(b))$, $\forall b \in \partial \mathcal{K}_h$
        \State $k \gets k + 1$
    \EndFor
  \end{algorithmic}
\end{algorithm}
The total number of training epochs is $N_{\text{SGD}} \times
N_{\text{Uz}}$.

\section{Numerical Results}\label{sec:Numresults}

In this section, we present numerical experiments to evaluate the
performance of our proposed methodologies, focusing on the
effectiveness of the Deep Uzawa approaches (RitUz and PINNUz) for
singularly perturbed boundary value problems.

For these experiments, we use the PyTorch Adam optimiser
\cite{stochasticADAMPyTorch} with a learning rate of $\eta =
10^{-3}$. Unless stated otherwise, we set $N_{\text{SGD}} = 40$ and
$N_{\text{Uz}} = 500$, and adopt the sigmoid-weighted linear unit
(SiLU) activation function \cite{Elfwing2018SiLU_ref}:
\begin{equation}
    \sigma(x) := \dfrac{x}{1 + e^{-x}}.
\end{equation}

\subsection{Example: Two-sided 1D Boundary Layer}\label{ex:Example_1D_2_BL}

As a benchmark, we consider a singularly perturbed problem on the
domain $\Omega = (0,1)$ with a small parameter $\epsilon > 0$. The
exact solution exhibits boundary layers, presenting a challenge for
standard penalty-based neural network approaches that often require
careful tuning of penalty weights for accurate results. This example
allows us to examine how varying the parameters $\rho$, $\gamma$, and
$\epsilon$ influences the performance of RitUz and PINNUz.

We seek $u^* \in \sobh1(\Omega)$ that satisfies
\begin{equation}
    \begin{split}
        - \epsilon \Delta u^* + u^* = 1 \quad \text{in } \Omega, \qquad  u^*(0) = u^*(1) = 0.
    \end{split}
\end{equation}
The exact solution is given by
\begin{equation}\label{eq:1D_exact_sol}
    u^*(x) = 1 - \dfrac{e^{(1-x)/\sqrt{\epsilon}} + e^{x/\sqrt{\epsilon}}}{e^{1/\sqrt{\epsilon}} + 1}, \quad \text{for } x \in \Omega.
\end{equation}
This setup provides a test case to evaluate how RitUz and PINNUz
schemes perform under varying conditions, demonstrating their
capability to handle singularly perturbed problems without extensive
tuning. The neural networks used in these experiments have depth $L=5$
and width $h=40$, parameterised by $\theta$.

\subsubsection{RitUz}
Figures \ref{fig:RitUz_1D_2BL} and \ref{fig:RitUz_1D_2BL_gamma_0} show
the results of applying the RitUz algorithm to the problem defined in
Example \ref{ex:Example_1D_2_BL}. In these experiments, we fix the PDE
parameter $\epsilon \in \{10^{-1}, 10^{-3}\}$ and the boundary
Lagrangian parameter $\gamma \in \{0, 2\}$. We vary the Uzawa
parameter $\rho \in [0.01, 20]$ to study its impact on convergence.

For both large and small values of $\epsilon$, we observe that the
rate of convergence of the $\leb2$-error with respect to the update
number increases as $\rho$ increases. This trend holds until $\rho$
exceeds the bound specified by condition
\eqref{eq:Dirichlet_energy_rho_assumption}, after which $u^k$ fails to
converge to $u^*$ as the update number increases. 



\begin{figure}[h!]
  \centering
  \includegraphics[width=\textwidth]{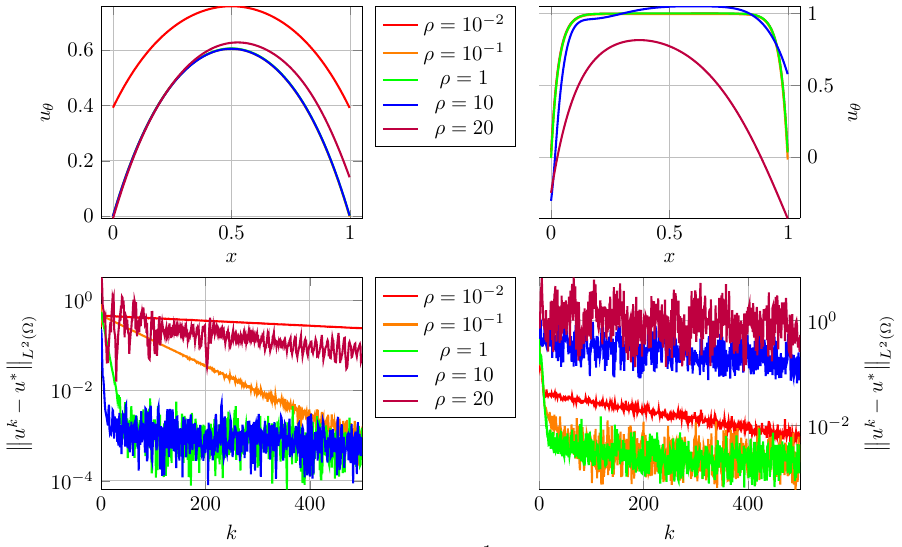} 
  \caption{Results for Example \ref{ex:Example_1D_2_BL} using the
    RitUz scheme with $\gamma = 2$ and $\epsilon = 10^{-1}$ (left) and
    $\epsilon = 10^{-3}$ (right). The plots show the state
    $u_{\theta}$ (top), and the $\leb2$-error
    $\Norm{u^k-u^*}_{\leb2(\W)}$ (bottom) as $\rho$ varies in $[0.01,
      20]$.}
  \label{fig:RitUz_1D_2BL}
\end{figure}

\begin{figure}[h!]
  \centering
  \includegraphics[width=\textwidth]{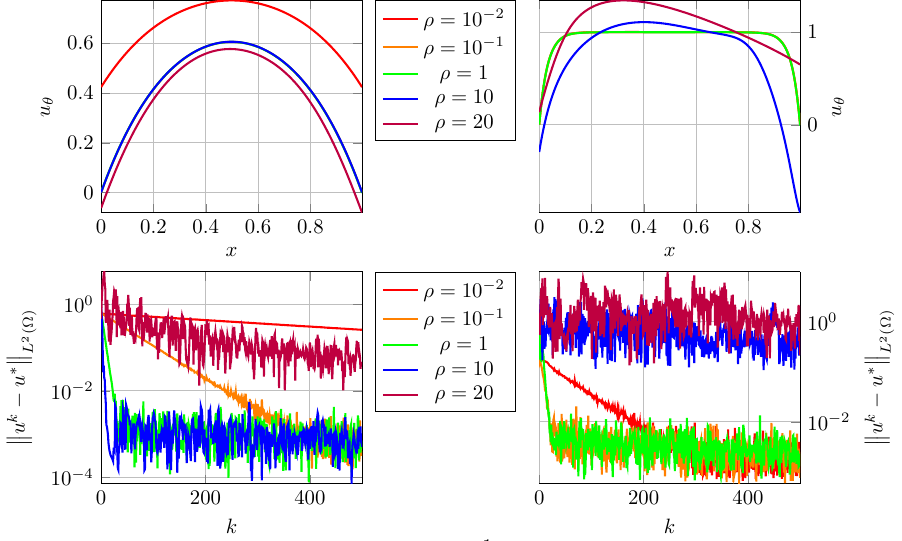}
  \caption{Results for Example \ref{ex:Example_1D_2_BL} using the
    RitUz scheme with $\gamma = 0$ and $\epsilon = 10^{-1}$ (left) and
    $\epsilon = 10^{-3}$ (right). The plots show the state
    $u_{\theta}$ (top), and the $\leb2$-error
    $\Norm{u^k-u^*}_{\leb2(\W)}$ (bottom) as $\rho$ varies in $[0.01,
      20]$.}
  \label{fig:RitUz_1D_2BL_gamma_0}
\end{figure}

\subsubsection{PINNUz}
Figure \ref{fig:PINNUz_1D_2BL} presents the results of the PINNUz
algorithm applied to the problem from Example
\ref{ex:Example_1D_2_BL}. In these experiments, we fix the PDE
parameter $\epsilon \in \{10^{-1}, 10^{-3}\}$ and set the boundary
Lagrangian parameter $\gamma = 2$. The Uzawa parameter $\rho$ is
varied in the range $[10^{-4}, 5]$ to examine its effect on
convergence.

As with the RitUz scheme, we observe that the $\leb2$-error
convergence rate with respect to the update number improves as $\rho$
increases. This holds until $\rho$ violates the condition in equation
\eqref{eq:PINNs_energy_gamma_assumption_no2}, causing $u^k$ to fail to
converge to $u^*$ as the update number grows. 


\begin{figure}[h!]
  \centering
  \includegraphics[width=\textwidth]{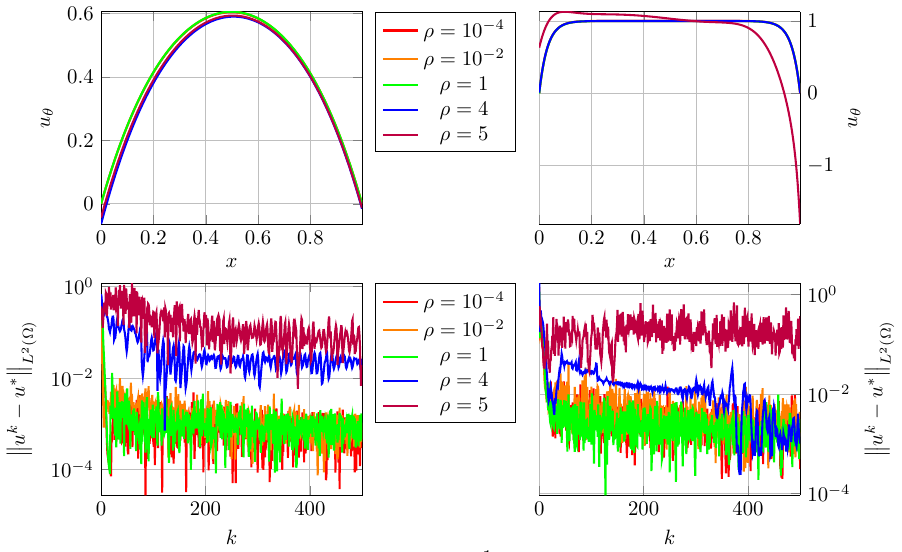}
  \caption{Results for Example \ref{ex:Example_1D_2_BL} using the
    PINNUz scheme with $\gamma = 2$ and $\epsilon = 10^{-1}$ (left)
    and $\epsilon = 10^{-3}$ (right). The plots show the state
    $u_{\theta}$ (top), and the $\leb2$-error
    $\Norm{u^k - u^*}_{\leb2(\W)}$ (bottom) for $\rho \in [10^{-4},
      5]$.}
  \label{fig:PINNUz_1D_2BL}
\end{figure}

\subsection{Example: 2D L-shaped Domain}\label{ex:Example_L_shaped}

As highlighted in Remark \ref{rem:nonconvex}, extending our
methodology to non-convex domains introduces additional challenges due
to geometric singularities. To demonstrate the robustness of the
PINNUz scheme in such settings, we consider a non-convex,
two-dimensional L-shaped domain $\W := (-1,1)^2 \setminus ([0,1)
  \times (-1,0])$. In this example, the solution $u^* \in \sobh1(\W)$
satisfies the reaction-diffusion equation:
\begin{equation}
    -\epsilon\Delta u^* + u^* = f \quad \text{in } \W, \quad u = g \quad \text{on } \partial \W.
\end{equation}
We choose $f$ and $g$ so that the exact solution is:
\begin{equation}\label{eq:L-shaped_exact_soln}
    \begin{split}
        u^*(x, y) = &\bra{x^2 + y^2}^{\frac{2}{3}} \sin\left(\dfrac{2}{3} \left[\text{atan2}(y, -x) - \pi\right]\right) \\
        &\times \bra{(x+1)^2 + (y-1)^2}^{\frac{2}{3}} \sin\left(2 \cdot \text{atan2}(y-1, x+1)\right).
    \end{split}
\end{equation}
It is chosen to have the correct asymptotic behaviour at the origin
whilst also non-trivial boundary conditions elsewhere. An illustration
of this solution is shown in Figure \ref{fig:L-shaped_exact_soln_a}.

For this experiment, we use $\eta = 10^{-4}$. We study the behaviour
of the networks for fixed values of $\epsilon$ and $\gamma$, while
varying $\rho$. The network architecture differs slightly from the
one-dimensional example, in that $L=10$ and $h=40$.

Figures
\ref{fig:L-shaped_exact_soln_b}--\ref{fig:L-shaped_exact_soln_c}
displays the state $u_\theta$ for different values of $\rho$: two
cases where the solution converges in the $\leb2$-norm and one where
it diverges. In the convergent cases, we observe that the
approximation struggles to match the boundary data near the re-entrant
corner at $(0,0)$, a common challenge in non-convex domains due to
reduced regularity. When $\rho$ exceeds the threshold defined in
equation \eqref{eq:PINNs_energy_gamma_assumption_no2}, the scheme
fails to converge globally.


\begin{figure}[h!]
    \centering
    \begin{subfigure}[t]{0.32\textwidth}
        \includegraphics[width=\textwidth]{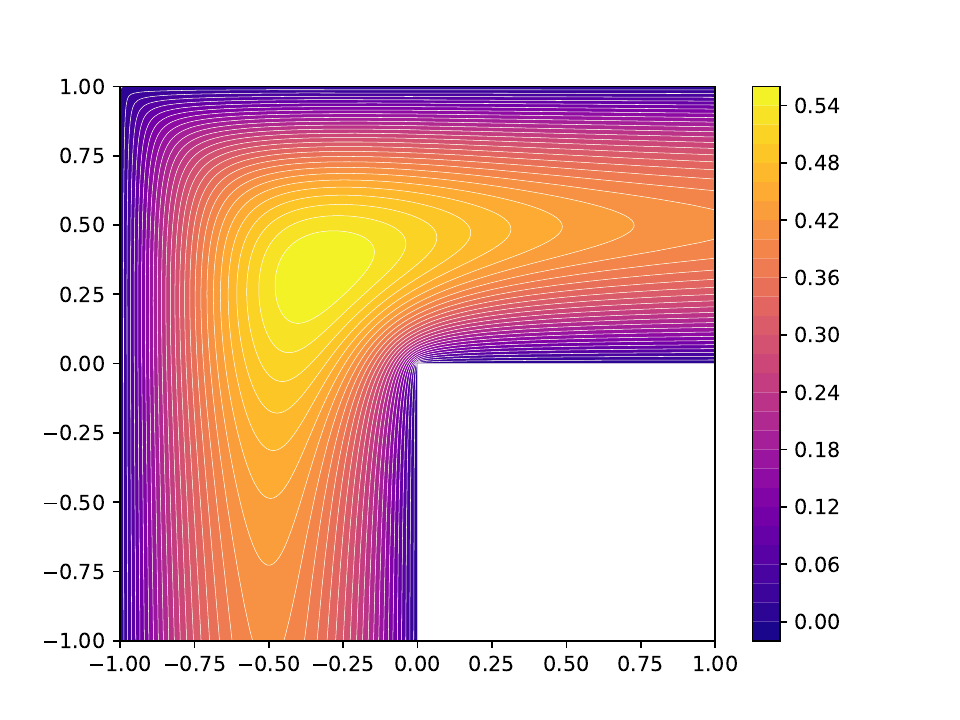}
        \caption{$u^*$}
        \label{fig:L-shaped_exact_soln_a}
    \end{subfigure}%
    \hfill
    \begin{subfigure}[t]{0.32\textwidth}
        \includegraphics[width=\textwidth]{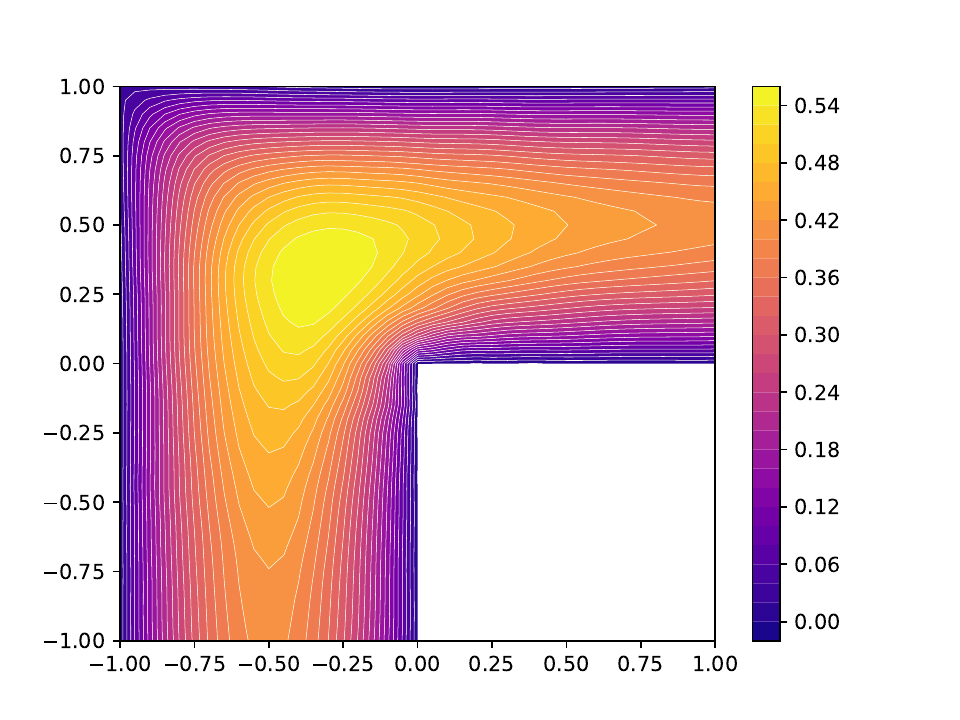}
        \caption{$u_\theta$ for $\rho = 1$}
        \label{fig:L-shaped_exact_soln_b}
    \end{subfigure}%
    \hfill
    \begin{subfigure}[t]{0.32\textwidth}
        \includegraphics[width=\textwidth]{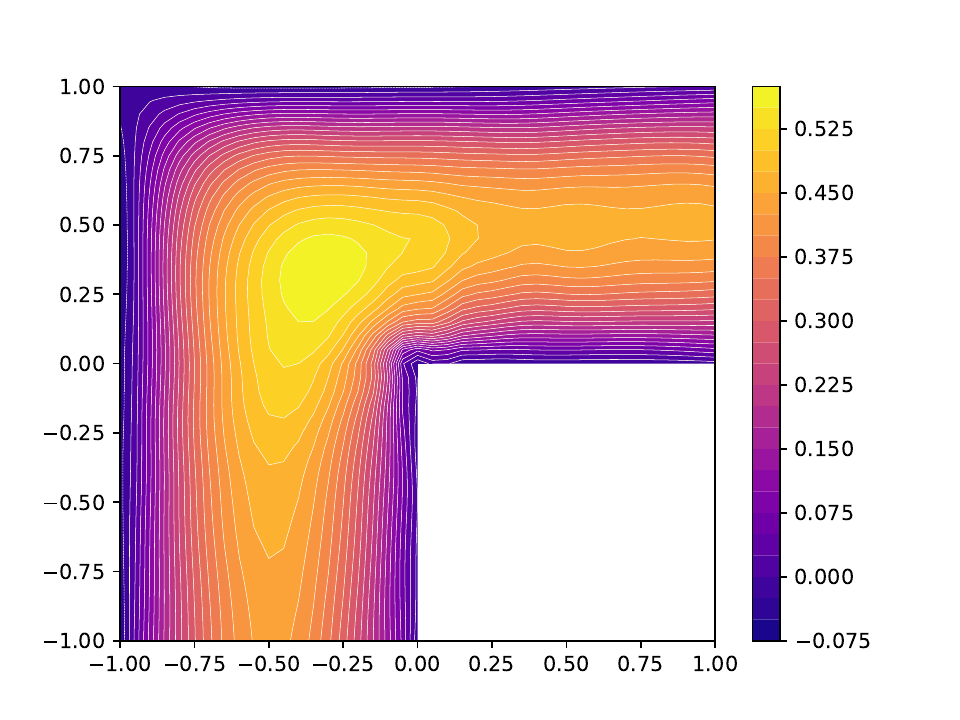}
        \caption{$u_\theta$ for $\rho = 5$}
        \label{fig:L-shaped_exact_soln_c}
    \end{subfigure}
    \caption{Example \ref{ex:Example_L_shaped}: (a) Illustration of the exact solution $u^*$ on the L-shaped domain as defined in equation \eqref{eq:L-shaped_exact_soln}. (b) and (c) show the state $u_\theta$ for $\gamma = 2$, $\epsilon = 10^{-3}$, and $\rho = 1$ and $\rho = 5$, respectively.}
    \label{fig:PINNUz_2D_combined}
\end{figure}

\begin{figure}
    \centering
    \includegraphics[width=\textwidth]{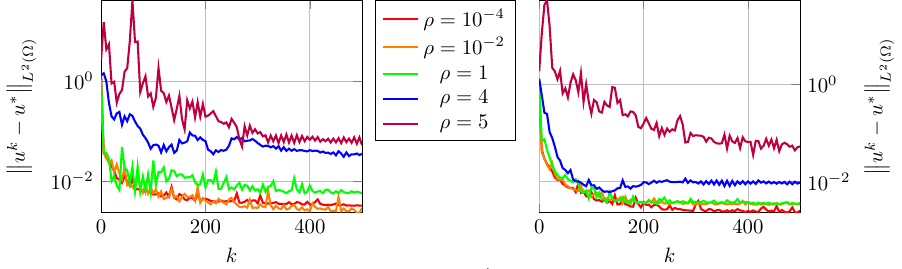}
    \caption{Example \ref{ex:Example_L_shaped}: The $\leb2$-error $\Norm{u^k-u^*}_{\leb2(\W)}$ for $\gamma = 2$, $\epsilon = 10^{-1}$ (left) and $\epsilon = 10^{-3}$ (right), with $\rho \in [10^{-4}, 5]$.}
    \label{fig:PINNUz_2D_L_error_adjoint}
\end{figure}

\subsection{Validating against penalty-based boundary}\label{sec:ValidationHardPenalty}

Penalty-based methods impose boundary conditions by minimising
\begin{equation}
  u^D_{\theta, \gamma} := \min\limits_{u_\theta \in \mathcal{V}_N} J_{D}(u_\theta),
  \qquad
  u^R_{\theta, \gamma} := \min\limits_{u_\theta \in \mathcal{V}_N} J_{R}(u_\theta),
\end{equation}
where $J_{D}$ and $J_{R}$ are defined in equations
\eqref{eq:DG_Cost_functional} and \eqref{eq:PINNs_Cost_functional},
respectively. The penalty parameter $\gamma$ controls the weight of
the boundary condition enforcement.

Penalty-based methods are flexible and straightforward to implement
but require careful tuning of $\gamma$. High values of $\gamma$ can
lead to better adherence to boundary conditions but may also cause
numerical stiffness and ill-conditioning, making optimisation more
difficult.

\subsubsection{RitUz and PINNUz}
Figures \ref{fig:DG_no_Uzawa} and \ref{fig:PINNs_no_Uzawa} illustrate
the performance of penalty-based methods for different $\gamma$
values. For moderate $\gamma$, the solution approximates the true
solution, but larger $\gamma$ can lead to issues with stiffness. The RitUz and PINNUz schemes, implemented with moderate $\rho$ values
and reduced $\gamma$, show that these methods can achieve reliable
convergence without extensive parameter tuning. 

\begin{figure}[h!]
    \centering
    \includegraphics[width=\textwidth]{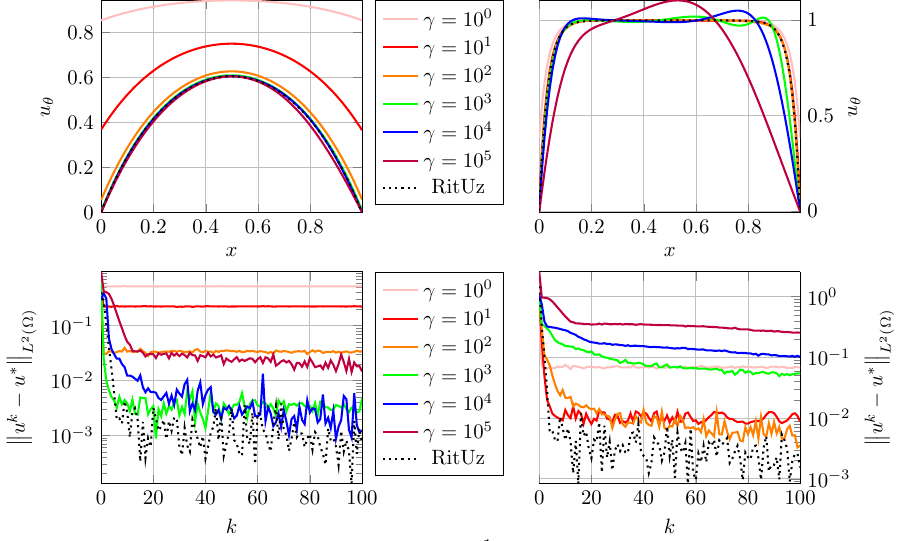}
    \caption{Example \ref{ex:Example_1D_2_BL}: For $\epsilon =
      10^{-1}$ (left) and $\epsilon = 10^{-3}$ (right), results from
      minimising $J_{D}$ over $\mathcal{V}_N$ for $\gamma \in [1,
        10^5]$. The plots show the state (top) and $\leb2$-error
      vs. update number (bottom). Outputs of the RitUz scheme for $\gamma = 0$
      and $\rho = 1$ are also shown.}
    \label{fig:DG_no_Uzawa}
\end{figure}

\begin{figure}[h!]
    \centering
    \includegraphics[width=\textwidth]{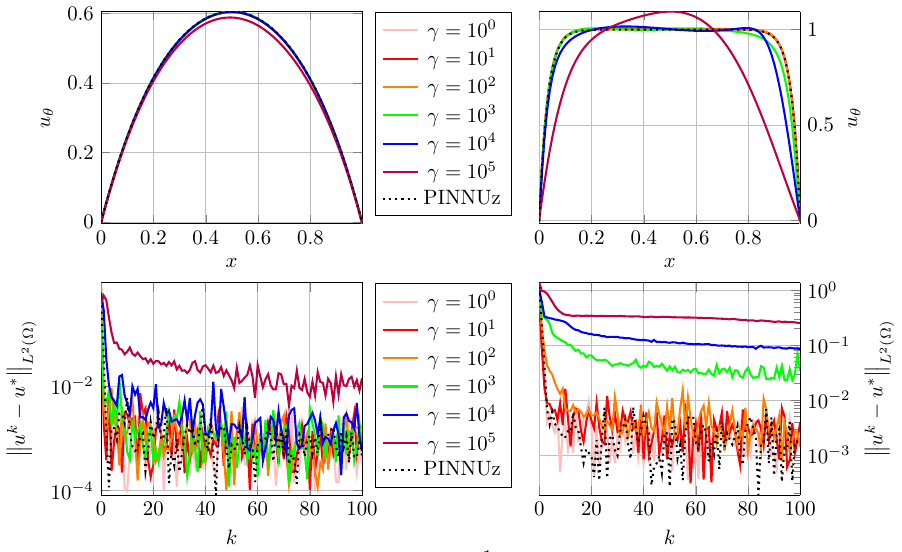}
    \caption{Example \ref{ex:Example_1D_2_BL}: For $\epsilon =
      10^{-1}$ (left) and $\epsilon = 10^{-3}$ (right), results from
      minimising $J_{R}$ over $\mathcal{V}_N$ for $\gamma \in [1,
        10^5]$. The plots show the state (top) and $\leb2$-error
      vs. update number (bottom). Outputs from the PINNUz scheme for $\gamma =
      2$ and $\rho = 1$ are also plotted.}
    \label{fig:PINNs_no_Uzawa}
\end{figure}

\subsection{Example: Higher-Dimensional Problems}\label{sec:ND_Laplace}

In this section, we consider solving Laplace's equation on the
$2d$-dimensional unit sphere, extending the problem to higher
dimensions. The harmonic function $u^* \in \sobh{1}(\Omega)$ satisfies
the Dirichlet boundary condition:
\begin{equation}\label{eq:Laplace_boundary_condition}
    u^*(\vec{x}) = 
        \sum\limits_{i=1}^{d} x_{2i-1} x_{2i}, 
    \quad \forall \vec{x} \in \mathbb{S}^{2d-1}.
\end{equation}
A common method for enforcing boundary conditions directly is through
hard imposition, which modifies the neural network's output to
inherently satisfy the boundary conditions
\cite{Sukumar22hard_bcs_citation, Yulei_Liao_2021,
  lu2021physicsinformedneuralnetworkshard}.

\subsubsection{RitUz}
We compare three methodologies for this problem: hard boundary
conditions (cRitz) \cite{Sukumar22hard_bcs_citation, Yulei_Liao_2021},
the Ritz penalty method, and the RitUz scheme.

The cRitz method conditions the neural network output as follows:
\begin{equation}
    (1 - \abs{\vec{x}}^2) u_{\theta}(\vec{x}) + \abs{\vec{x}}^2 u^*(\vec{x}) \mapsto u_{\theta}(\vec{x}).
\end{equation}
The penalty method includes a standard $\leb2(\partial \W)$ penalty,
and the RitUz scheme follows the iterative approach previously
described.

For these experiments, collocation points are uniformly sampled within
the domain every 10 epochs, with a batch size of 1024 for the
Dirichlet energy computation. Boundary points are sampled with a batch
size of 2048, updated every 10 epochs. In the Uzawa scheme, an initial
fixed set of 2048 boundary points is used for approximating
$\lambda$. We employ a network depth of $L=5$ and width $h=40$.

Figure \ref{fig:Hrd_v_Pen_v_Uzw} shows the $\leb2$-error of the Ritz
methods per epoch, indicating that the RitUz method with $\gamma=10$
achieves significantly lower errors compared to the penalty method and
performs similarly to the hard boundary condition approach. The
increase in error with dimension suggests that $\rho$ may need to be
adjusted due to the dimensional dependency of $C_{\text{tr}}$, as
discussed in Remark \ref{rem:trace}.

\begin{figure}[h!]
    \centering
    \includegraphics[width=\textwidth]{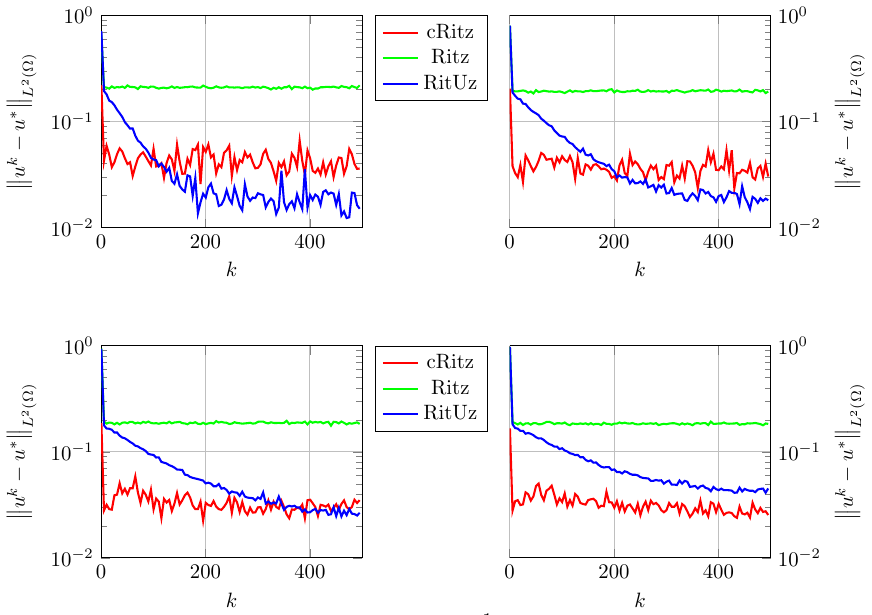}
    \caption{Example \ref{sec:ND_Laplace}: Comparison of
      $\leb2$-errors for the cRitz, penalty, and RitUz methods in
      dimensions 4 (top left), 6 (top right), 8 (bottom left), and 10
      (bottom right) with $\gamma=10$ and $\rho=0.1$.}
    \label{fig:Hrd_v_Pen_v_Uzw}
\end{figure}

\begin{figure}[h!]
    \centering
    \includegraphics[width=0.55\textwidth]{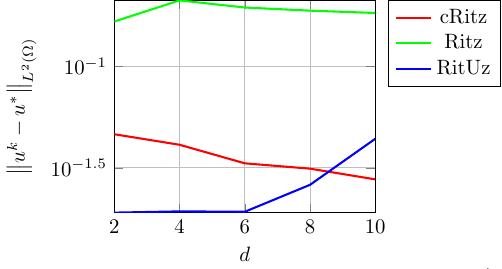}
    \caption{Example \ref{sec:ND_Laplace}: Final $\leb2$-errors for
      the cRitz, penalty, and RitUz methods after 50,000 epochs for
      dimension $d=2,4,6,8,10$.}
    \label{fig:Hrd_v_Pen_v_Uzw_final}
\end{figure}

\subsubsection{PINNUz}\label{sec:ND_PINNs_Laplace}

We extend the experiment to PINNs, comparing hard boundary conditions
(hPINNs) \cite{lu2021physicsinformedneuralnetworkshard}, penalty
PINNs, and the PINNUz scheme. The network architecture and sampling
methodology are as in the previous section.

Figure \ref{fig:Error_v_EpUp} shows the error variation with respect
to $N_{\text{SGD}}$ for the PINNUz scheme. Higher dimensions require
more epochs per Uzawa step to maintain accurate updates. For
$N_{\text{SGD}}=500$, performance matches the hard boundary condition
for 2D and 4D problems, but errors increase for higher dimensions,
suggesting that $N_{\text{SGD}}$ may need adjustment.

Figure \ref{fig:Error_v_EpUp_PINNUz} presents the final $\leb2$-errors
for hPINNs, PINNs and the PINNUz scheme across dimension. The figure
highlights that for lower-dimensional problems (2D and 4D), the PINNUz
scheme performs comparably to hard boundary condition methods,
achieving similar error levels.

Figure \ref{fig:time_v_dim_PINNUz} shows that the computation time
scales approximately linearly with problem dimension across hPINNs,
PINNs, and PINNUz.

\begin{figure}[h!]
    \centering
    \includegraphics[width=.55\textwidth]{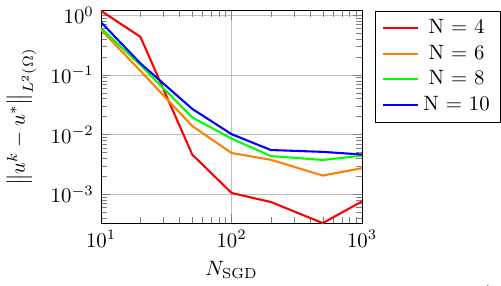}
    \caption{Example \ref{sec:ND_PINNs_Laplace}: Error of the PINNUz scheme versus $N_{\text{SGD}}$ for $\gamma=2$, $\rho=0.1$, and various dimensions.}
    \label{fig:Error_v_EpUp}
\end{figure}

\begin{figure}[h!]
    \centering
    \includegraphics[width=\textwidth]{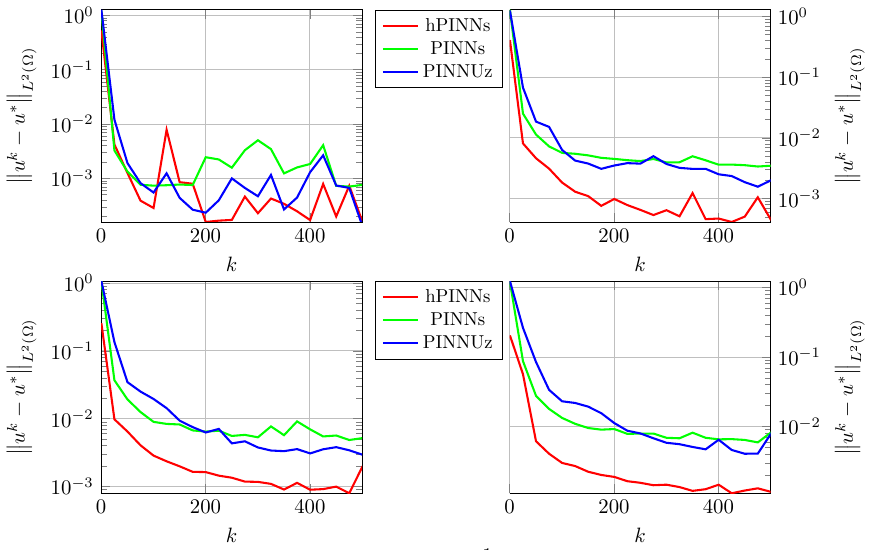}
    \caption{Example \ref{sec:ND_PINNs_Laplace}: $\leb2$-errors for hard boundary conditions, penalty methods, and the PINNUz scheme for dimensions 2, 4, 6, and 8, averaged over three trials.}
    \label{fig:Error_v_EpUp_PINNUz}
\end{figure}

\begin{figure}[h!]
    \centering
    \includegraphics[width=\textwidth]{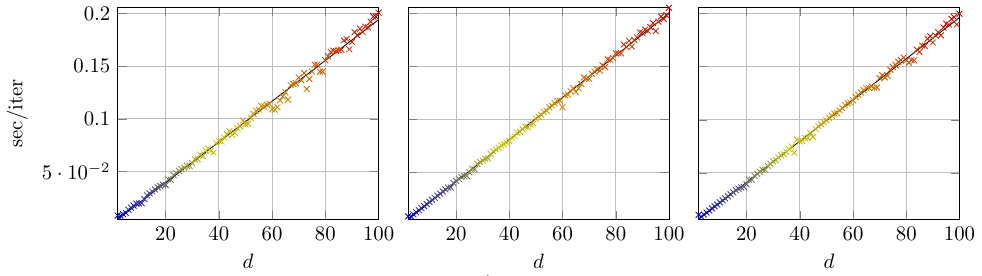}
    \caption{Example \ref{sec:ND_PINNs_Laplace}: Average time per iteration against dimension $d =2,\ldots, 100$; for hard boundary conditions (left), penalty methods (centre), and the Uzawa method (right).}
    \label{fig:time_v_dim_PINNUz}
\end{figure}

\section*{Acknowledgements}
This work was initiated under a during a hosted visit at ITE,
Crete. AP and TP received support from the EPSRC programme grant
EP/W026899/1. TP was also supported by the Leverhulme RPG-2021-238 and
EPSRC grant EP/X030067/1.

%% file: nnText.tex
\begin{tikzpicture}[scale=1.2, transform shape]

\foreach \i in {1, 2, 3} {
    \node[circle, draw=black, fill=white, minimum size=0.7cm] (I\i) at (0, -\i) {};
}

\foreach \i in {1, 2, 3, 4} {
    \node[circle, draw=black, fill=white, minimum size=0.7cm] (H1\i) at (2, -\i+0.5) {};
}

\foreach \i in {1, 2, 3, 4} {
    \node[circle, draw=black, fill=white, minimum size=0.7cm] (H2\i) at (4, -\i+0.5) {};
}

\foreach \i in {1, 2, 3, 4} {
    \node[circle, draw=black, fill=white, minimum size=0.7cm] (H3\i) at (6, -\i+0.5) {};
}


\foreach \i in {1} {
    \node[circle, draw=black, fill=white, minimum size=0.7cm] (O\i) at (8, -2) {};
}
\node[scale=0.8] at (8., -2) {$u$};

\foreach \i in {1, 2, 3} {
    \foreach \j in {1, 2, 3, 4} {
        \draw[->] (I\i) -- (H1\j);
    }
}

\foreach \i in {1, 2, 3, 4} {
    \foreach \j in {1, 2, 3, 4} {
        \draw[->] (H1\i) -- (H2\j);
    }
}

\foreach \i in {1, 2, 3, 4} {
    \foreach \j in {1, 2, 3, 4} {
        \draw[->] (H2\i) -- (H3\j);
    }
}

\foreach \i in {1, 2, 3, 4} {
    \foreach \j in {1} {
        \draw[->] (H3\i) -- (O\j);
    }
}

\node[scale=0.8] at (0, 0.5) {Input Layer};
\node[scale=0.8] at (2, 0.5) {$\mathcal{C}_1 \circ \sigma$};
\node[scale=0.8] at (4, 0.5) {$\mathcal{C}_2 \circ \sigma$};
\node[scale=0.8] at (6, 0.5) {$\mathcal{C}_3 \circ \sigma$};
\node[scale=0.8] at (8, 0.5) {Output Layer};

\end{tikzpicture}